\documentclass[11pt]{amsart}
\usepackage{amsmath,amsthm,amssymb, amscd, amsfonts}
\usepackage[all]{xy}
\usepackage{leftidx}
\usepackage[latin1]{inputenc}
\usepackage{enumerate}
\usepackage[dvipsnames]{xcolor}

\theoremstyle{plain}

\newtheorem{theorem}{Theorem}[section]
\newtheorem{corollary}[theorem]{Corollary}

\newtheorem{lemma}[theorem]{Lemma}

\theoremstyle{definition}

\theoremstyle{remark}
\newtheorem{remark}[theorem]{Remark}

\numberwithin{equation}{section}\theoremstyle{plain}

\newcommand{\C}{{\mathcal C}}
\newcommand{\D}{{\mathcal D}}

\newcommand{\E}{{\mathcal E}}

\newcommand{\Rep}{\operatorname{Rep}}

\newcommand\Irr{\operatorname{Irr}}
\newcommand\FPdim{\operatorname{FPdim}}

\newcommand\vect{\operatorname{Vect}}

\newcommand\id{\operatorname{id}}

\newcommand\End{\operatorname{End}}

\begin{document}

\title[Classification of semisimple Hopf algebras]{On semisimple quasitriangular Hopf algebras of dimension $dq^n$}

\author[Dong]{Jingcheng Dong}
\email[Dong]{dongjc@njau.edu.cn}

\author[Dai]{Li Dai}
\email[Dai]{daili1980@njau.edu.cn}



\date{\today}


\begin{abstract}
Let $q>2$ be a prime number, $d$ be an odd square-free natural number, and $n$ be a non-negative integer. We prove that a semisimple quasitriangular Hopf algebra of dimension $dq^n$ is solvable in the sense of Etingof, Nikshych and Ostrik. In particular, if $n\leq 3$ then it is either isomorphic to $k^G$ for some abelian group $G$, or twist equivalent to a Hopf algebra which fits into a cocentral abelian exact sequence.
\end{abstract}

\maketitle



\section{Introduction}\label{intro}
Let $H$ be a finite dimensional Hopf algebra, and $D(H)$ be its Drinfeld double. Then $D(H)$ is a quasitriangular Hopf algebra, and there is a Hopf algebra inclusion $H\hookrightarrow D(H)$. Hence, the classification of finite dimensional quasitriangular Hopf algebras can be viewed as a first step towards the classification of all finite-dimensional Hopf algebras.

Let $(H,R)$ be a quasitriangular Hopf algebra. Assume that ${\rm dim}H$ is square-free and odd. Natale proved that $H$ is semisimple and is a group algebra \cite{natale2006r-matrices}. Furthermore, Natale \cite{bruguieres2011exact} proved that a braided fusion category whose Frobenius-Perron dimension is odd and square-free is equivalent to the category of representations of a semisimple quasitriangular Hopf algebra, and hence to $\Rep(G)$ for some finite group $G$. The present paper is devoted to extend these results. The main technique used in this paper is the theory of fusion categories.  Our main result is the theorem below.

\begin{theorem}\label{thm00}
Let $q>2$ be a prime number, $d$ be an odd square-free natural number, and $n$ be a non-negative integer. Then

(1)\, A semisimple quasitriangular Hopf algebra of dimension $dq^n$ is solvable.

\medbreak
If $n\leq 3$ then

(2)\, A braided fusion category of Frobenius-Perron dimension $dq^n$ is equivalent to the category of representations of a semisimple quasitriangular Hopf algebra.

(3)\, A semisimple quasitriangular Hopf algebra of dimension $dq^n$ is either isomorphic to $k^G$ for some abelian group $G$, or twist equivalent to a Hopf algebra which fits into a cocentral abelian exact sequence.
\end{theorem}

Let $q$ be a prime number and let $d$ be a square-free natural number. In \cite{2016DongIntegral} and \cite{2016DongNatale} we call a fusion category almost square-free if its Frobenius-Perron dimension is $dq^n$. In this paper we also call a Hopf algebra almost square-free if its dimension is $dq^n$.

This paper is organized as follows. In section 2, we recall some notions and basic results which will be used throughout. In section 3, we study the solvability of an almost square-free semisimple quasitriangular Hopf algebra with odd dimension. We also prove that this class of Hopf algebras have nontrivial $1$-dimensional representations. In section 4, we study the case when $n\leq 3$, and obtain the classification result.

Throughout, we will work over an algebraically closed field $k$ of characteristic $0$. For a finite group $G$, $kG$ denotes the group algebra of $G$ over $k$, and $k^G$ denotes the dual group algebra of $kG$. All Hopf algebras considered in this paper are finite dimensional over $k$. Our reference for Hopf algebras is \cite{1993Montgomery} and the reference for the theory of tensor (fusion) categories is \cite{egno2015}.

\section{Preliminaries}\label{prels}
\subsection{Quasitriangular Hopf algebras}
A quasitriangular structure in a Hopf algebra $H$ is an invertible element $R\in H\otimes H$, called an $R$-matrix, satisfying:

 (QT1)\, $(\Delta\otimes \id)(R)=R_{13}R_{23}$;

 (QT2)\, $(\id\otimes \Delta)(R)=R_{13}R_{12}$;

 (QT3)\, $(\varepsilon\otimes \id)(R)=1=(\id\otimes\varepsilon)(R)$;

 (QT4)\, $\Delta^{cop}(h)=R\Delta(h)R^{-1},\forall h\in H$.

 In this case, $(H,R)$ is called a quasitriangular Hopf algebra. A quasitriangular Hopf algebra $(H,R)$ is called factorizable if the map $\Phi_{R}:H^*\to H$, given by $\Phi_{R}(f)=<f,Q^{(1)}>Q^{(2)}$ for $f\in H^*$, is an isomorphism, where $Q=Q^{(1)}\otimes Q^{(2)}=R_{21}R\in H\otimes H$.

\medbreak
The $R$-matrix defines a braided structure on the category $\Rep(H)$ of finite dimensional representations of $H$. It turns out that $\Rep(H)$ is a braided tensor category. In particular, if $(H,R)$ is a semisimple factorizable Hopf algebra then $\Rep(H)$ is a non-degenerate braided fusion category \cite{Turaer1994}.

\subsection{Exact sequences of Hopf algebras}
An exact sequence of finite dimensional Hopf algebras is a sequence of Hopf algebra maps
$$k\to K\xrightarrow{i}H\xrightarrow{\pi}L\to k,$$

such that

(1)\, $i$ is injective and $\pi$ is surjective;

(2)\, $\pi\circ i=\varepsilon_K1$, where $\varepsilon_K$ is the counit of $K$;

(3)\, ${\rm ker}\pi=HK^+$.

If $K$ is commutative and $L$ is cocommutative then this exact sequence is called \emph{abelian}. If this is the case then there exist finite groups $\Sigma$ and $\Gamma$ such that $K\cong k^{\Sigma}$ and $L\cong k\Gamma$, and we have
\begin{equation}\label{eq201}
\begin{split}
k\to k^{\Sigma}\xrightarrow{i}H\xrightarrow{\pi}k\Gamma\to k.
\end{split}
\end{equation}
$H$ is also called a Kac algebra in this case.

\medbreak
If a Hopf algebra $H$ fits into an exact sequence (\ref{eq201}) then $H$ is isomorphic to a bicrossed product ${k^{\Sigma}}^{\tau}\#_{\sigma}k\Gamma$ as a Hopf algebra with respect to some normalized $2$-cocycles $\tau$ and $\sigma$.

An exact sequence (\ref{eq201}) is called central if $k^{\Sigma}$ is in the center of $H$; while it is called cocentral if its dual is central.

\begin{remark}\label{rem0}
By \cite[Lemma 3.3]{natale2008hopf}, if the exact sequence (\ref{eq201}) is central then $\sigma$ is trivial; if it is cocentral then $\tau$ is trivial.
\end{remark}

An exact sequence of Hopf algebras is a basic tool in the construction of Hopf algebras, and also plays an important role in the classification of Hopf algebras, see \cite{natale1999semisimple} and \cite{etingof2011weakly} for examples.

\subsection{Pointed fusion categories}\label{sec22}
Let $\C$ be a fusion category. We use $\Irr(\C)$ to denote the set of isomorphism classes of simple objects of $\C$. The Frobenius-Perron dimension $\FPdim(X)$ of a simple object $X\in\C$ is the Frobenius-Perron eigenvalue of the matrix of left multiplication by the class of $X$ in $\Irr(\C)$. The Frobenius-Perron dimension of $\C$ is the number $$\FPdim(\C)=\sum_{X\in\Irr(\C)}\FPdim(X)^2.$$

A fusion category $\C$ is called weakly integral if $\FPdim(\C)$ is an integer. If $\FPdim(X)$ is an integer for every $X\in\Irr(\C)$ then $\C$ is called integral.

The simplest integral fusion category is a pointed fusion category in which every simple object has Frobenius-Perron dimension $1$. If $\C$ is a pointed fusion category, then $\C$ is equivalent to the category $\vect_G^{\omega}$  of $G$-graded vector spaces with associativity constraint given by a $3$-cocycle $\omega\in H^3(G,k^{\times})$ \cite{Ostrik2003}. In particular, if $\omega=1$ then $\C$ is equivalent to the category of representations of $k^G$.

\subsection{Braided fusion categories}\label{sec23}
A braided fusion category $\C$ is a fusion category equipped with a natural isomorphism $c_{X,Y}:X\otimes Y\to Y\otimes X$, for all objects $X,Y \in \C$, satisfying the hexagon diagrams \cite{kassel1995quantum}.

A balancing transformation (or a twist) of a braided fusion category $\C$ is a natural automorphism $\theta: \id_{\C}\to \id_{\C}$ satisfying
$$\theta_{\textbf{1}} = \id_{\textbf{1}} \mbox{\, and \,} \theta_{X\otimes Y}=(\theta_X\otimes \theta_Y)\circ c_{Y,X}\circ c_{X,Y}.$$

A braided fusion category $\C$ is called premodular if it admits a twist $\theta$ satisfying $\theta_X^*=\theta_{X^*}$ for every $X\in \C$.


\medbreak
Let $\C$ be a braided fusion category. If $\D$ is a fusion subcategory of  $\C$, the M\" uger centralizer $\D'$ of $\D$ in $\C$ is the full fusion subcategory generated by the set
$$\{X \in \C| c_{Y, X}c_{X, Y} =\id_{X \otimes Y}, \forall Y \in \D\}.$$

The M\"uger center of $\C$ is the M\" uger centralizer $\C'$. The braided fusion category is called non-degenerate if $\C'\cong \vect$ is trivial. A premodular category $\C$ is called modular if it is non-degenerate.

\begin{remark}\label{rem1}
In the following context, we only consider weakly integral fusion categories. Any weakly integral braided fusion category has a canonical twist \cite{etingof2005fusion}. Hence, an weakly integral braided fusion category is modular if it is non-degenerate.
\end{remark}

\medbreak
A category $\C$ is called \emph{symmetric} if $\C' = \C$. Let $G$ be a finite group. The fusion category  $\Rep(G)$ of finite dimensional representations of $G$ is a symmetric fusion category with respect to the canonical braiding. A braided fusion category $\E$ is called Tannakian if $\E \cong \Rep(G)$ for some finite group $G$, as symmetric fusion categories.

\medbreak
A fiber functor for a tensor category $\C$ over $k$ is a tensor functor $F: \C\to \vect$. By the \emph{reconstruction theorem for finite dimensional Hopf algebras}, given a fiber functor $F$ for a tensor category $\C$ over $k$, the algebra $H:=\End(F)$ has a natural structure of a finite dimensional Hopf algebra, and we have a canonical equivalence of tensor categories $\C\cong \Rep(H)$. Conversely, given a finite dimensional Hopf algebra $H$, the forgetful functor $F:\Rep(H)\to \vect$ is a fiber functor.

The lemma below is modified from \cite[Lemma 7.3]{bruguieres2011exact} which we restate for the convenience of the reader.

\begin{lemma}\label{lem1}
Let $\C$ be a braided pointed category whose Frobenius-Perron dimension is an odd integer $N$. Then there exists an abelian group $G$ of order $N$ such that $\C$ is equivalent to the category of $G$-graded vector spaces $\vect_G^1$. In particular, $\C$ has a fiber functor.
\end{lemma}
\begin{proof}
Since $\C$ is pointed, $\C=\vect_G^{\omega}$ for some finite group $G$ and some cohomology class $\omega\in H^3(G,k^{\times})$.

The braided and premodular structures on a pointed fusion category are classified in \cite[Section 7.5]{frohlich1993} in terms of cohomology. In particular, for a given twist $\theta$ on $\C$, the cohomology class $\omega$ is trivial if and only if $\theta$ is $1$ on the subgroup ${}_2G=\{g\in G:g^2=1\}$. Since the order of $G$ is odd, ${}_2G$ must be trivial and hence $\omega$ is trivial. Finally, the existence of a braiding implies that $G$ is abelian.
\end{proof}

\medbreak
Let $\C$ be a braided fusion category, and $\Rep(G)$ be a Tannakian subcategory of $\C$. Following the procedure described in \cite[Section 4.4]{drinfeld2010braided}, we can get a new fusion category $\C_G$, called the de-equivariantization of $\C$ by $\Rep(G)$. Conversely, the fusion category $\C_G$ admits an action of $G$ by tensor autoequivalences, and this action gives rise to a new fusion category $(\C_G)^G$ such that $\C\cong (\C_G)^G$. By \cite[Proposition 4.26]{drinfeld2010braided}, we have
\begin{equation}\label{eq112}
\begin{split}
\FPdim(\C_G)=\frac{\FPdim(\C)}{|G|}.
\end{split}
\end{equation}

Let $\Rep(G)$ be a Tannakian subcategory of $\C$. Then the de-equivariantization $\C_G$ of $\C$ by $\Rep(G)$ admits a $G$-grading $\C_{G}=\oplus_{g\in G}(\C_G)_g$. The trivial component $(\C_G)_e$ of this grading is also a braided fusion category. The category $\C$ is non-degenerate if and only if $(\C_G)_e$ is non-degenerate and the grading is faithful, that is, $\C_g\neq 0$ for all $g\in G$ \cite[Proposition 4.56]{drinfeld2010braided}.

\medbreak
The notion of an exact sequence of tensor categories was introduced and studied in \cite{bruguieres2011exact}. By definition, an exact sequence of tensor categories is a diagram of tensor functors $\C'\to \C\to \C''$ satisfying certain conditions. In particular, given a Tannakian subcategory $\Rep(G)$ of a braided fusion category $\C$, we have an exact sequence of fusion categories (see \cite[Section 1]{bruguieres2011exact}):
$$\Rep(G)\to \C\to \C_G,$$
where $\C_G$ is the de-equivariantization of $\C$ by $\Rep(G)$.

\section{Solvability of a semisimple quasitriangular Hopf algebra}
In this section, we assume that $q$ is a prime number, $d$ is a square-free natural number and $n$ is a non-negative integer. We also assume that $\C$ is a braided fusion category of Frobenius-Perron dimension $dq^n$.

The notion of solvability of a fusion category was introduced in \cite{etingof2011weakly}. In this paper, we call a semisimple Hopf algebra solvable if its representation category is solvable.

\begin{lemma}\label{lem401}
Assume that $\C$ is non-degenerate. Then $\C$ is solvable.
\end{lemma}
\begin{proof}
Assume first that $\C$ is integral. By \cite[Lemma 3.4]{2016DongIntegral}, $\C$ has a Tannakian subcategory equivalent to $\Rep(\mathbb{Z}_q)$. Let $\C_{\mathbb{Z}_q}$ be the de-equivariantization of $\C$ by $\Rep(\mathbb{Z}_q)$. Then $\C_{\mathbb{Z}_q}=\oplus_{g\in\mathbb{Z}_q}(\C_{\mathbb{Z}_q})_g$ is a grading of $\C_{\mathbb{Z}_q}$. Since $\C$ is non-degenerate, this grading is faithful and the trivial component $(\C_{\mathbb{Z}_q})_e$ is also non-degenerate. In addition, $(\C_{\mathbb{Z}_q})_e$ has Frobenius-Perron dimension $dq^{n-2}$ by equations (\ref{eq112}). By induction on $n$, $(\C_{\mathbb{Z}_q})_e$ is a solvable fusion category. Because the class of solvable fusion categories is closed under taking extensions and equivariantizations by solvable groups \cite[Proposition 4.5]{etingof2011weakly}, $\C_{\mathbb{Z}_q}$ and hence $\C$ are both solvable.

We then assume that $\C$ is not integral. In this case, $\C$ is a $G$-extension of an integral fusion subcategory $\D$, where $G$ is an elementary abelian $2$-group \cite[Theorem 3.10]{gelaki2008nilpotent}. Since $\C$ is non-degenerate, \cite[Lemma 1.2]{etingof1998some} shows that the square of $\FPdim(X)$ divides $\FPdim(\C)$, for every $X\in \Irr(\D)$. In addition, $\D$ is an integral braided fusion category. All these facts show that $\FPdim(X)$ is a power of $q$. So $\D$ is solvable by \cite[Corollary 3.5]{dong2014existence}, see also \cite[Theorem 7.2]{natale2013weakly}. It follows that $\C$ is solvable, since it is an extension of a solvable fusion category by a solvable group.
\end{proof}

\begin{remark}\label{rem402}
This result was also obtained by Natale in \cite{natale2013weakly} by different method.
\end{remark}

\begin{theorem}\label{thm403}
Assume that $\C$ is degenerate and has odd dimension. Then $\C$ is solvable.
\end{theorem}
\begin{proof}
Since $\C$ is degenerate and $\FPdim(\C)$ is odd, the M\"{u}ger center $\C'$ of $\C$ is a Tannakian subcategory by \cite[Corollary 2.50]{drinfeld2010braided}. So there is a finite group $G$ such that $\C'\cong \Rep(G)$. In addition, $G$ is a solvable group since its order is odd. Let $\C_{G}$ be the de-equivariantization of $\C$ by $\Rep(G)$. By \cite[Remark 2.3]{etingof2011weakly}, $\C_{G}$ is a non-degenerate fusion category. The Lemma \ref{lem401} shows that $\C_{G}$ is a solvable fusion category. It follows that $\C$, being an equivariantization of a solvable fusion category by a solvable group, is a solvable fusion category.
\end{proof}

\begin{remark}\label{rem404}
Note that a degenerate braided fusion category of Frobenius-Perron dimension $dq^n$ which is even may not be solvable. For example, let $\mathbb{A}_5$ be the alternating group of order $60=2^2\times 3\times 5$, and let $\Rep(\mathbb{A}_5)$ be the category of finite dimensional representations of $\mathbb{A}_5$. Since $\mathbb{A}_5$ is a simple group, $\Rep(\mathbb{A}_5)$ is not solvable by \cite[Proposition 4.5]{etingof2011weakly}.
\end{remark}

\begin{corollary}\label{cor405}
Let $H$ be a semisimple quasitriangular Hopf algebra of dimension $dq^n$. Assume that $\FPdim(\C)$ is odd. Then $H$ is solvable.
\end{corollary}
\begin{proof}
Let $\C=\Rep(H)$ be the category of representations of $H$. Then $\C$ is a braided fusion category. If $H$ is factorizable then $\C$ is non-degenerate, and hence the corollary follows from Lemma \ref{lem401}. If $H$ is not factorizable then $\C$ is degenerate, and hence the corollary follows from Theorem \ref{thm403}.
\end{proof}

\begin{corollary}\label{cor406}
Let $H$ be a semisimple quasitriangular Hopf algebra of dimension $dq^n$. Assume that $\FPdim(\C)$ is odd. Then $H$ has a nontrivial $1$-dimensional representation.
\end{corollary}
\begin{proof}
Let $\C=\Rep(H)$ be the category of representations of $H$. Then $\C$ is solvable by Corollary \ref{cor405}. By \cite[Proposition 4.5]{etingof2011weakly}, $\C$ contains a nontrivial invertible object $\delta$. Let $\D$ be the fusion subcategory generated by $\delta$. Then $\D\cong \Rep(k^{\Sigma})$ for some cyclic group $\Sigma$. Hence we have a Hopf quotient $H\to k^{\Sigma}$. This implies that $k\Sigma\subseteq kG(H^*)\subseteq H^*$ which completes the proof.
\end{proof}

\section{Classification of a semisimple quasitriangular Hopf algebra}\label{mainResults}
\begin{theorem}\label{thm1}
Let $\C$ be a braided fusion category. Suppose that $\C'$ contains a Tannakian subcategory $\Rep(G)$ such that the de-equivarization $\C_G$ of $\C$ by $\Rep(G)$ is pointed and has odd dimension. Then $\C$ is equivalent to the category of representations of a semisimple quasitriangular Hopf algebra.
\end{theorem}
\begin{proof}
By assumption, we have an exact sequence of fusion categories
$$\Rep(G)\to\C\to\C_G.$$

Since $\Rep(G)$ is contained in the M\"{u}ger center $\C'$, the de-equivarization $\C_G$ is braided by \cite[Remark 2.3]{etingof2011weakly}. By Lemma \ref{lem1}, $\C_G$ has a fiber functor, and hence $\C$ has a fiber functor. By the reconstruction theorem for finite dimensional Hopf algebras, $\C$ is equivalent to the category of representations of a Hopf algebra $H$. Since $\C$ is braided and semisimple, $H$ is quasitriangular and semisimple.
\end{proof}

The notion of a group-theoretical fusion category was introduced in \cite{etingof2005fusion}. By \cite[Theorem 7.2]{naidu2009fusion}, a braided fusion category is group-theoretical if and only if it contains a Tannakian subcategory such that the corresponding de-equivariantization is pointed. Hence the fusion category $\C$ in Theorem \ref{thm1} is a special case of group-theoretical fusion categories. The following corollary shows that this special case really exists. When $n=1$, the following result has been proved by Brugui{\`e}res and Natale in \cite{bruguieres2011exact}.

\begin{corollary}\label{cor31}
Let $q>2$ be a prime number and $d$ be an odd square-free natural number. Suppose that $\C$ is a braided fusion category of Frobenius-Perron dimension $dq^n$ with $n\leq 3$. Then $\C$ is equivalent to the category of representations of a semisimple quasitriangular Hopf algebra.
\end{corollary}
\begin{proof}
Since the Frobenius-Perron dimension of $\C$ is odd, \cite[Corollary 3.11]{gelaki2008nilpotent} shows that $\C$ is an integral fusion category.

Suppose that $\C$ is non-degenerate. Then $\C$ is pointed by \cite[Corollary 3.3]{2016DongIntegral} and the theorem follows from Lemma \ref{lem1}.

Suppose that $\C$ is degenerate.  Then the M\"{u}ger center $\C'$ of $\C$ is not trivial. Since $\FPdim\C$ is odd, \cite[Corollary 2.50]{drinfeld2010braided} shows that $\C'\cong\Rep(G)$ is a Tannakian subcategory for some finite group $G$. Let $\C_G$ be the de-equivarization of $\C$ by $\Rep(G)$. By equation (\ref{eq112}), $\FPdim(\C_G)=d'q^{n'}$, where $d'$ is a square-free natural number dividing $d$ and $n'\leq 3$. By \cite[Corollary 4.31]{drinfeld2010braided}, $\C_G$ is a non-degenerate braided fusion category. Hence $\C_G$ is pointed, also by \cite[Corollary 3.3]{2016DongIntegral}. By Theorem \ref{thm1}, $\C$ is equivalent to the category of representations of a quasitriangular semisimple Hopf algebra.
\end{proof}

\medbreak
 A twist in a finite dimensional Hopf algebra $H$ is an invertible element $J\in H\otimes H$ such that
\begin{align*}
(\varepsilon\otimes \id)(J)&=(\id\otimes \varepsilon)(J)=1,
\\ (\Delta\otimes \id)(J)(J\otimes 1)&=(\id\otimes\Delta)(J)(1\otimes J).
\end{align*}

If $J\in H\otimes H$ is a twist then there is a new Hopf algebra $(H^J,\Delta^J,S^J)$, where $H^J=H$ as an algebra, comultiplication $\Delta^J(h)=J^{-1}\Delta(h)J$ and antipode $S^J(h)=v^{-1}S(h)v$ where $v=m(S\otimes \id)(J)$.

Two Hopf algebras $H$ and $H'$ are called twist equivalent if $H'\cong H^J$ for some twist $J$. It is well-known that $H$ and $H'$ are twist equivalent if  and only if $\Rep(H)$ is equivalent to $\Rep(H')$ as tensor category. Therefore, the class of semisimple or quasitriangular Hopf algebras is closed under twist.

\begin{corollary}\label{cor32}
Let $q>2$ be a prime number and $d$ be an odd square-free natural number. Suppose that $H$ is a semisimple quasitriangular Hopf algebra of dimension $dq^n$ with $n\leq 3$. Then

(1)\, $H$ is isomorphic to $k^G$ for some abelian group $G$;

(2)\, $H$ is twist equivalent to a Hopf algebra $H'$  which fits into a cocentral abelian exact sequence
\begin{equation}\label{eq30}
\begin{split}
k\to k^\Sigma\to H'\to k\Gamma\to k,
\end{split}
\end{equation}
where $\Sigma$ and $\Gamma$ are finite groups, and $\Sigma$ is abelian.
\end{corollary}

\begin{proof}
Let $\C=\Rep(H)$ be the category of finite dimensional representations of $H$. Then $\C$ is a braided fusion category. If $H$ is factorizable then $\C$ is non-degenerate, and hence it is pointed \cite[Corollary 3.3]{2016DongIntegral}. Then part (1) follows from Lemma \ref{lem1}.

Assume that $H$ is not factorizable then $\C$ is degenerate. By the proof of Corollary \ref{cor32}, the M\"{u}ger center of $\C$ is a Tannakian subcategory. Hence there exists a finite group $\Gamma$ with odd order such that $\Rep(\Gamma)$ is the M\"{u}ger center of $\C$. Let $\C_{\Gamma}$ be the de-equivarization of $\C$ by $\Rep(\Gamma)$. Then $\C\cong (\C_{\Gamma})^{\Gamma}$ is an $\Gamma$-equivarization of $\C_{\Gamma}$ by the discussion in Subsection \ref{sec23}. Also by the proof of Corollary \ref{cor31}, $\C_{\Gamma}$ is pointed and has a fiber functor, and hence $\C_{\Gamma}$ is equivalent to $\Rep(k^{\Sigma})$ for some abelian group $\Sigma$ by Lemma \ref{lem1}. Therefore, $\C\cong (\Rep(k^{\Sigma}))^{\Gamma}$. It follows from \cite[Proposition 4.2]{natale2011semisimple} that $H$ is twist equivalent to a Hopf algebra $H'$ which fits into a cocentral abelian exact sequence as described.
\end{proof}

\begin{remark}\label{rem2}
When $n=1$, one more precise classification result is obtained by Natale \cite{natale2006r-matrices}. Let $H$ be a quasitriangular Hopf algebra. Assume that ${\rm dim}H$ is odd and square-free. Natale proves that $H$ is semisimple and isomorphic to a group algebra.

Since the class of quasitriangular Hopf algebras is closed under twist, the Hopf algebra $H'$ in Corollary \ref{cor32} is also quasitriangular. In addition, $H'$ is a bicrossed product ${k^{\Sigma}}^{\tau}\#_{\sigma}k{\Gamma}$ for some $\tau$ and $\sigma$. Because the exact sequence (\ref{eq30}) is cocentral, $H'$ is a crossed product ${k^{\Sigma}}\#_{\sigma}k{\Gamma}$ as an algebra by \cite[Lemma 3.3]{natale2008hopf}.
\end{remark}

\section{Acknowledgements}
The research of the authors was partially supported by the Fundamental Research Funds for the Central Universities (KYZ201564), the Natural Science Foundation of China (11571173, 11201231) and the Qing Lan Project.



\end{document}